\documentclass[12pt]{article}

\usepackage{amssymb}
\usepackage{amsthm}
\usepackage{amsmath}
\usepackage{graphicx}
\usepackage{fullpage}
\usepackage{color}
\usepackage{enumerate}
 \numberwithin{equation}{section}
 \usepackage{caption}
\theoremstyle{plain}
\newtheorem{thm}{Theorem}[section]
\newtheorem{cor}[thm]{Corollary}

\newtheorem{lem}[thm]{Lemma}
\newtheorem{prop}[thm]{Proposition}

\theoremstyle{definition}
\newtheorem{defn}[thm]{Definition}
\newtheorem{ex}[thm]{Example}

\newtheorem*{theorem*}{Theorem}
\newtheorem*{defn*}{Definition}

\theoremstyle{remark}

\newtheorem{rem}[thm]{Remark}

\newcommand{\N}{\mathbb{N}}
\newcommand{\R}{\mathbb{R}}

\newcommand{\bp}{\begin{proof}[\ensuremath{\mathbf{Proof}}]}
\newcommand{\bs}{\begin{proof}[\ensuremath{\mathbf{Solution}}]}
\newcommand{\ep}{\end{proof}}

\newcommand{\tr}{\operatorname{tr}}

\begin{document}

\title{Extremal functions for Morrey's inequality \\ in convex domains}

\author{Ryan Hynd\footnote{Department of Mathematics, MIT. Partially supported by NSF grants DMS-1301628 and DMS-1554130 and an MLK visiting professorship.}\; and Erik Lindgren\footnote{Department of Mathematics, Uppsala University. Supported by the Swedish Research Council, grant no. 2012-3124 and 2017-03736.}}  

\maketitle

\begin{abstract}
For a bounded domain $\Omega\subset \R^n$ and $p>n$, Morrey's inequality implies that there is $c>0$ such that 
$$
c\|u\|^p_{\infty}\le \int_\Omega|Du|^pdx
$$
for each $u$ belonging to the Sobolev space $W^{1,p}_0(\Omega)$.  We show that the ratio of any two extremal functions is constant provided that $\Omega$ is convex. We also show with concrete examples why this property fails to hold in general and verify that convexity is not a necessary condition for a domain to have this feature. As a by product, we obtain the uniqueness of an optimization problem involving the Green's function for the $p$-Laplacian.

\end{abstract}


\section{Introduction}
Suppose $\Omega\subset \R^n$ is a bounded domain and $p>n$.  Morrey's inequality for $W^{1,p}_0(\Omega)$ functions $u$ may be expressed as
$$
c\|u\|^p_{C^{1-\frac{n}{p}}(\overline\Omega)}\le \int_\Omega|Du|^pdx,
$$
where $c>0$ is a constant that is independent of $u$. In particular,  
\begin{equation}\label{WeakMorrey}
c\|u\|^p_{\infty}\le  \int_\Omega|Du|^pdx.
\end{equation}
Let us define
$$
\lambda_p:=\inf\left\{\raisebox{-0.5 em}{$\displaystyle\frac{\displaystyle \int_\Omega|Du|^pdx}{\displaystyle\| u\| _\infty^p}$}:u\in W_0^{1,p}(\Omega)\setminus\{0\}\right\}.
$$
Observe that
\begin{equation}\label{SecondMorrey}
\lambda_p\|u\|^p_{\infty}\le \int_\Omega|Du|^pdx
\end{equation}
and that $c=\lambda_p$ is the largest constant such that \eqref{WeakMorrey} is valid. Furthermore, if there is a function $u \in W^{1,p}_0(\Omega)\setminus\{0\}$ such that equality holds in \eqref{WeakMorrey}, then $c=\lambda_p$.
\begin{defn*}
A function $u\in W^{1,p}_0(\Omega)\setminus\{0\}$ is an {\it extremal} if equality holds in \eqref{SecondMorrey}. 
\end{defn*}

\par It is plain to see that any multiple of an extremal is also an extremal. Using routine compactness arguments, it is not difficult to verify that extremal functions exist. We will argue below that any extremal $u$ satisfies the boundary value problem 
\begin{equation}\label{pGroundStateMeasure}
\begin{cases}
-\Delta_pu=\lambda_p |u(x_0)|^{p-2}u(x_0)\delta_{x_0}\quad &x\in \Omega\\
\hspace{.34in} u=0 \quad &x\in \partial \Omega,
\end{cases}
\end{equation}
which was derived by Ercole and Pereira in \cite{ErcoleP}. Here $\Delta_p\psi:=\text{div}(|D\psi|^{p-2}D\psi)$ is the $p$-Laplacian, and $x_0$ is the {\it unique} point for which $|u|$ is maximized in $\Omega$.  Moreover, using \eqref{pGroundStateMeasure} we will be able to conclude that any extremal has a definite sign in $\Omega$. And as the PDE in \eqref{pGroundStateMeasure} is homogeneous, the optimal constant $\lambda_p$ can be interpreted as being an eigenvalue.

\par The primary goal of this work is to address the extent to which extremal functions can be different. In particular, we would like to know if any two extremal functions are necessarily multiples of one another. If they are, we consider the set of extremals to be uniquely determined. For once one extremal is found, all others can be obtained by scaling. We will argue that annuli never have this uniqueness property.  We will also exhibit star-shaped domains for which this uniqueness property fails. However, we will see that if a planar domain has certain symmetry, then its extremals are one dimensional.   

\par Our main result is that convex domains always have the aforementioned uniqueness property. 
\begin{thm}\label{mainThm}
Assume that $\Omega\subset \R^n$ is open, convex and bounded. If $u$ and $v$ are extremal, then $u/v$ is constant throughout $\Omega$. 
\end{thm}
\noindent We will also explain how Theorem \ref{mainThm} implies the following corollary involving the Green's function of the $p$-Laplacian in $\Omega$. 

\begin{cor}\label{Greenthm} Assume that $\Omega\subset \R^n$ is open, convex and bounded. Suppose that $G(\cdot,y)$ is the Green's function of the $p$-Laplacian in $\Omega$ with pole $y\in \Omega$; that is, $G(\cdot,y)$ satisfies 
\begin{equation}\label{GreenPDE}
\begin{cases}
-\Delta_p w=\delta_{y}\quad &x\in\Omega\\
\hspace{.34in}w=0\quad &x\in \partial\Omega.
\end{cases}
\end{equation}
Then $x_0$ given in equation \eqref{pGroundStateMeasure} is the unique point in $\Omega$ for which
$$
G(x_0,x_0)=\max\{G(y,y): y\in \Omega\}.
$$
\end{cor}

\par Part of our motivation was to extend a previous result of Talenti. He considered extremal functions for the following inequality, which is also due to Morrey. For each weakly differentiable function $u:\R^n\rightarrow \R$,
\begin{equation}\label{MorreyonWholeSpace}
\|u\|^p_{L^\infty(\R^n)}
\le c(n,p)\left|\text{supp}(u)\right|^{\frac{p}{n}-1}\int_{\R^n}|Du|^pdx.
\end{equation}
Here $c(n,p)$ is an explicit constant depending only on $p$ and $n$, and $|\text{supp}(u)|$ is the Lebesgue measure of the support of $u$. Employing Schwarz symmetrization, Talenti showed in \cite{Talenti} that if equality holds in \eqref{MorreyonWholeSpace} there are $a\in \R$, $r>0$ and $x_0\in \R^n$ such that
$$
u(x)=
\begin{cases}
a\left(r^{\frac{p-n}{p-1}}-|x-x_0|^{\frac{p-n}{p-1}}\right), \quad & |x-x_0|< r\\
0, & |x-x_0|\ge r.
\end{cases}
$$
We remark that a quantitative version of this result has been established by Cianchi \cite{Cianchi}, and we refer the reader to \cite{Cianchi89},  \cite{EF}, and \cite{Tal87} for work on sharp constants of related inequalities.

\par Unfortunately, $\R^n$ and balls are the only known domains for which the extremals have such convenient characterizations. Nevertheless, in this paper, we believe that we have taken significant steps in understanding precisely which domains have a one dimensional collection of extremals. In Section \ref{Prelim}, we will derive basic properties of solutions of \eqref{pGroundStateMeasure}, and in Section \ref{SuppSec}, we consider the support function of an extremal.  In Section \ref{CounterEx}, we will provide examples of domains for which uniqueness fails; these include annuli, bow tie and dumbbell shaped planar domains.  In Section \ref{UniquenessArgs}, we verify Theorem \ref{mainThm}, and in Section \ref{Steiner}, we use Steiner symmetrization to exhibit some nonconvex planar domains that have unique extremals.

\section{Properties of extremals}\label{Prelim}
We now proceed to deriving some properties of extremal functions. These properties will be crucial to our uniqueness study. First, we verify that extremal functions satisfy the boundary value problem \eqref{pGroundStateMeasure}. Then we will study the behavior of solutions of \eqref{pGroundStateMeasure} near their global maximum or minimum points. We also refer the reader to the  recent paper \cite{ErcoleP} by 
Ercole and Pereira, where they studied properties of extremal functions in a wide class of inequalities that include \eqref{WeakMorrey}. In particular, they obtained analogous results to Corollary \ref{sign} and Corollary \ref{sol} below.
 
\begin{lem}
A function $u\in W_0^{1,p}(\Omega)$ is extremal if and only if  
\begin{equation}\label{WeakMeasure}
\int_\Omega|Du|^{p-2}Du\cdot D\phi dx=\lambda_p\max\left\{|u(x)|^{p-2}u(x)\phi(x): x\in \overline{\Omega}, |u(x)|=\|u\|_\infty\right\}
\end{equation}
for all $\phi\in W_0^{1,p}(\Omega)$.
\end{lem}
\begin{proof}
1. First let us establish the following identity 
\begin{equation}\label{supnormderiv}
\lim_{\epsilon\rightarrow0^+}\frac{\|u+\epsilon\phi \|^p_\infty-\|u\|^p_\infty}{\epsilon}=p\max\left\{|u(x)|^{p-2}u(x)\phi(x): x\in \overline{\Omega}, |u(x)|=\|u\|_\infty\right\}
\end{equation}
for $u,\phi\in C(\overline\Omega)$.  For any $x_0$ such that $|u(x_0)|=\|u\|_\infty$,
\begin{align*}
\frac{1}{p}\|u+\epsilon\phi\|_\infty^p&\ge\frac{1}{p}|u(x_0)+\epsilon\phi(x_0)|^p\\
&\ge \frac{1}{p}|u(x_0)|^p+\epsilon|u(x_0)|^{p-2}u(x_0)\phi(x_0)\\
&= \frac{1}{p}\|u\|_\infty^p+\epsilon|u(x_0)|^{p-2}u(x_0)\phi(x_0).
\end{align*}
Therefore, 
$$
\liminf_{\epsilon\rightarrow0^+}\frac{\|u+\epsilon\phi \|^p_\infty-\|u\|^p_\infty}{\epsilon}\ge p|u(x_0)|^{p-2}u(x_0)\phi(x_0)
$$
and so $``\ge"$ holds in \eqref{supnormderiv}. 
\par Now choose a sequence of positive numbers $(\epsilon_j)_{j\in \N}$ tending to $0$ such that 
$$
\limsup_{\epsilon\rightarrow0^+}\frac{\|u+\epsilon\phi\|_\infty^p-\|u\|_\infty^p}{\epsilon}=\limsup_{j\rightarrow\infty}\frac{\|u+\epsilon_j\phi\|_\infty^p-\|u\|_\infty^p}{\epsilon_j}
$$
and select a sequence $(x_{j})_{j\in \N}$ maximizing $|u+\epsilon_j \phi|$ that converges to a maximizer $x_0$ of $|u|$. Such sequences exist by the continuity of $u$ and $\phi$, the compactness of $\overline\Omega$, and the inequalities $|u(x)+\epsilon_j \phi(x)|\le \|u+\epsilon_j \phi\|_\infty=|u(x_j)+\epsilon_j\phi(x_j)|$. As $\R\ni z\mapsto \frac{1}{p}|z|^p$ is continuously differentiable,
\begin{align*}
\limsup_{j\rightarrow\infty}\frac{\|u+\epsilon_j\phi\|_\infty^p-\|u\|_\infty^p}{\epsilon_j}&\le \limsup_{j\rightarrow\infty}\frac{|u(x_{j})+\epsilon_j\phi(x_{j})|^p-|u(x_{j})|^p}{\epsilon_j}\\
&=p|u(x_0)|^{p-2}u(x_0)\phi(x_0)\\
&\le p\max\{|u(x)|^{p-2}u(x)\phi(x) : |u(x)|=\|u\|_\infty \}.
\end{align*}
We conclude $``\le"$ in \eqref{supnormderiv}. 
\par 2. Any extremal $u\in W^{1,p}_0(\Omega)\setminus\{0\}$ satisfies 
$$
\lambda_p=\frac{\displaystyle\int_\Omega|Du|^pdx}{\| u\| _\infty^p}\le \frac{\displaystyle\int_\Omega|Du+\epsilon D\phi|^pdx}{\| u+\epsilon\phi \| _\infty^p}
$$
for each $\phi\in W^{1,p}_0(\Omega)$ and $\epsilon>0$ sufficiently small. Exploiting \eqref{supnormderiv}
\begin{align*}
0&\le \lim_{\epsilon\rightarrow0^+}\frac{1}{p\epsilon}\left(\frac{\displaystyle\int_\Omega|Du+\epsilon D\phi|^pdx}{\| u+\epsilon\phi \| _\infty^p}-\frac{\displaystyle\int_\Omega|Du|^pdx}{\| u\| _\infty^p}\right)\\
&= \frac{\displaystyle\int_\Omega|Du|^{p-2}Du\cdot D\phi dx}{\| u\| _\infty^p}-\frac{\displaystyle\int_\Omega|Du|^pdx}{\| u\| _\infty^{2p}}\max\{|u(x)|^{p-2}u(x)\phi(x): x\in \overline{\Omega}, |u(x)|=\|u\|_\infty\}\\
&=\frac{1}{\| u\| _\infty^p}\left(\int_\Omega|Du|^{p-2}Du\cdot D\phi dx-\lambda_p\max\{|u(x)|^{p-2}u(x)\phi(x): x\in \overline{\Omega}, |u(x)|=\|u\|_\infty\}\right).
\end{align*}
Canceling the factor $1/\| u\| _\infty^p$ and replacing $\phi$ with $-\phi$ gives \eqref{WeakMeasure}. 
\par 3. Of course if \eqref{WeakMeasure} holds, we can choose $\phi=u$ to verify that $u$ is extremal. 
\end{proof}

\begin{cor}\label{sign}
Each extremal function is everywhere positive or everywhere negative in $\Omega$. 
\end{cor}
\begin{proof}
Assume that $u\in W^{1,p}_0(\Omega)\setminus\{0\}$ is extremal. Then $w:=|u|$ is extremal, as well. Moreover, \eqref{WeakMeasure} implies 
$$
\int_\Omega|Dw|^{p-2}Dw\cdot D\phi dx\ge 0
$$
for all $\phi\ge 0$.  Therefore, $w$ is $p$-superharmonic, $w\ge 0$ and $w|_{\partial\Omega}=0$. Since $w$ doesn't vanish identically, $w=|u|>0$ 
(Theorem 11.1 in \cite{PucSer}). Hence, 
$u$ doesn't vanish in $\Omega$ and so $u$ has a definite sign in $\Omega$. 
\end{proof}
Observe that the left hand side of \eqref{WeakMeasure} is linear in $\phi$, while the right hand side appears to be nonlinear in $\phi$.  We will argue that this forces 
the set $\{x\in \overline{\Omega}: |u(x)|=\|u\|_\infty\}$ to be a singleton for any extremal function.

\begin{prop}
Assume $u$ is an extremal function. Then $\{x\in \Omega: |u(x)|=\|u\|_\infty\}$ is a singleton. 
\end{prop}
\begin{proof}
Without any loss of generality, we may assume $u>0$ and $\|u\|_\infty=1$. In view of \eqref{WeakMeasure},
\begin{equation}\label{linearRelation}
\max_{\{u=1\}}\{\phi_1+\phi_2\}=\max_{\{u=1\}}\phi_1+\max_{\{u=1\}}\phi_2
\end{equation}
for any two $\phi_1,\phi_2\in C^\infty_c(\Omega)$.  Suppose that there are distinct points $x_1$ and $x_2$ for which $u(x_1)=u(x_2)=1$. In this case, 
there are balls $B_\delta(x_1), B_\delta(x_2)\subset\Omega$ that are disjoint for some $\delta>0$ small enough. We choose functions $\phi_1\in C^\infty_c(\Omega)$, $\phi_2\in C^\infty_c(\Omega)$ that 
are nonnegative, have maximum value 1, and are supported in $B_\delta(x_1)$ and $B_\delta(x_2)$, respectively. It follows that
$$
\max_{\{u=1\}}\{\phi_1+\phi_2\}=1,
$$
while $\max_{\{u=1\}}\phi_1=1$ and $\max_{\{u=1\}}\phi_2=1$. This contradicts \eqref{linearRelation}.
\end{proof}

\begin{cor}\label{sol}
Assume $u$ is an extremal function. Then $|u|$ attains its maximum value uniquely at some $x_0\in\Omega$. Moreover, 
$$
\int_\Omega|Du|^{p-2}Du\cdot D\phi dx=\lambda_p |u(x_0)|^{p-2}u(x_0)\phi(x_0)
$$
for each $\phi\in W^{1,p}_0(\Omega)$. In particular, $u$ is a weak solution of \eqref{pGroundStateMeasure}.
\end{cor}

\par We note that any solution $u$ of \eqref{pGroundStateMeasure}  is differentiable with a locally H\"older continuous gradient in $\Omega\setminus\{x_0\}$, see \cite{Evans, Lewis1, Ural}.   However, we show below that $u$ is not differentiable at $x_0$. 

\begin{ex}\label{TalentiExample}
As we noted above, when $\Omega=B_r(x_0)$, we have an explicit extremal function 
\begin{equation}\label{TalentiLoc}
u(x)=a\left(r^{\frac{p-n}{p-1}}-|x-x_0|^{\frac{p-n}{p-1}}\right), \quad x\in B_r(x_0)
\end{equation}
for each $a\in \R$. Moreover, any extremal is of the form \eqref{TalentiLoc} for some $a\in \R$; in particular, any ball has a one dimensional collection of extremal functions. The corresponding optimal constant in \eqref{SecondMorrey} is
$$
\lambda_p=\displaystyle\frac{\displaystyle\int_{B_r(x_0)}|Du|^pdx}{\displaystyle\|u\|^p_\infty} =\left(\frac{p-n}{p-1}\right)^{p-1}r^{n-p}n\omega_n.
$$
\end{ex}
\begin{figure}[h]
\begin{center}
\includegraphics[width=.8\textwidth]{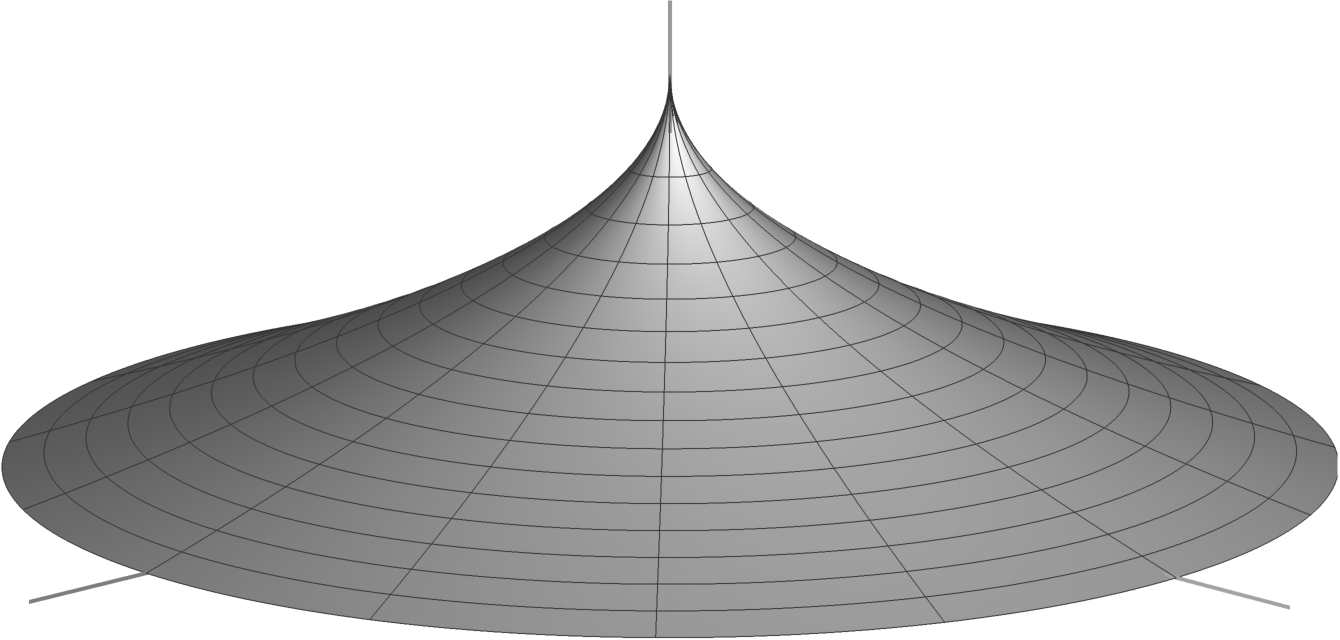}
\caption{Example \ref{TalentiExample} with $a=1$, $p=4$, $n=2$, $r=1$ and $x_0=0$.}
\end{center}
\end{figure}

\par We can use the extremals for balls \eqref{TalentiLoc} to study the behavior of general extremals near the points which maximize their absolute values. Note in particular, that the family of extremals \eqref{TalentiLoc} are H\"{o}lder continuous with exponent $\frac{p-n}{p-1}\in (0,1]$, which is a slight improvement of the exponent $\frac{p-n}{p}$ one has from the Sobolev embedding $W^{1,p}_0(\Omega)\subset C^{1-\frac{n}{p}}(\overline{\Omega})$.  We will first argue that solutions of \eqref{pGroundStateMeasure}, and in particular extremals, have exactly this type of continuity at their maximizing or minimizing points. 

\begin{prop}\label{holderprop}
Assume $u$ is a solution of \eqref{pGroundStateMeasure} and that $B_r(x_0)\subset \Omega\subset B_R(x_0)$. Then 
\begin{equation}\label{HolderEst}
\frac{\|u\|_\infty}{R^{\frac{p-n}{p-1}}}|x-x_0|^\frac{p-n}{p-1}\le |u(x)-u(x_0)|\le \frac{\|u\|_\infty}{r^{\frac{p-n}{p-1}}}|x-x_0|^\frac{p-n}{p-1}
\end{equation}
for each $x\in \Omega.$
\end{prop} 
\begin{proof}
Without loss of generality we may assume that $u>0$ and $u(x_0)=1$. Define 
\begin{align*}
v(x):&=\frac{1}{r^{\frac{p-n}{p-1}}}\left(r^{\frac{p-n}{p-1}}-|x-x_0|^{\frac{p-n}{p-1}}\right)\\
&=1-\frac{1}{r^{\frac{p-n}{p-1}}}|x-x_0|^{\frac{p-n}{p-1}}
\end{align*}
for $x\in B_r(x_0).$  Observe that $u$ and $v$ are $p$-harmonic in $B_r(x_0)\setminus\{x_0\}$, $u(x_0)=v(x_0)$ and $u\ge v$ on $\partial B_r(x_0)$. 
By weak comparison, $u\ge v$ in  $B_r(x_0)$. That is
$$
u(x)\ge 1-\frac{1}{r^{\frac{p-n}{p-1}}}|x-x_0|^{\frac{p-n}{p-1}}, \quad x\in B_r(x_0).
$$
Since $v(x)\le 0$ for $x\not\in B_r(x_0)$, the above inequality trivially holds for $x\in\Omega\setminus B_r(x_0)$. 

\par Now set 
\begin{align*}
w(x):&=\frac{1}{R^{\frac{p-n}{p-1}}}\left(R^{\frac{p-n}{p-1}}-|x-x_0|^{\frac{p-n}{p-1}}\right)\\
&=1-\frac{1}{R^{\frac{p-n}{p-1}}}|x-x_0|^{\frac{p-n}{p-1}}.
\end{align*}
Observe that $u$ and $w$ are $p$-harmonic in $\Omega\setminus\{x_0\}$, $u(x_0)=w(x_0)$ and $u\le w$ on $\partial\Omega$ as $\Omega\subset B_R(x_0)$. 
By weak comparison, $u\le w$ in  $\Omega$. That is
$$
u(x)\le 1-\frac{1}{R^{\frac{p-n}{p-1}}}|x-x_0|^{\frac{p-n}{p-1}}, \quad x\in \Omega.
$$
\end{proof}
\begin{cor}
Suppose that $u$ is a non-zero solution of \eqref{pGroundStateMeasure}. Then $u$ is not differentiable 
at $x_0$. 
\end{cor}
\begin{proof} 
First assume $\frac{p-n}{p-1}\in (0,1)$.  By hypothesis, $u(x)=u(x_0)+Du(x_0)\cdot(x-x_0)+o(|x-x_0|)$, as $x\rightarrow x_0$. Choosing $R$ so large that $\Omega\subset B_R(x_0)$, we have by the previous proposition that
$$
\frac{\|u\|_\infty}{R^{\frac{p-n}{p-1}}}|x-x_0|^\frac{p-n}{p-1}\le|u(x)-u(x_0)|\le |Du(x_0)||x-x_0|+o(|x-x_0|),
$$
as $x\rightarrow x_0$. That is, 
$$
\frac{\|u\|_\infty}{R^{\frac{p-n}{p-1}}}\frac{1}{|x-x_0|^{1-\frac{p-n}{p-1}}}\le |Du(x_0)|+o(1).
$$
This inequality can not be true since $\frac{p-n}{p-1}\in (0,1)$. If $\frac{p-n}{p-1}=1$, then $n=1$ and the claim trivially holds since $u$ is then of the form \eqref{TalentiLoc}.
\end{proof}
We will now refine the above estimates to deduce the exact behavior of a solution $u$ of \eqref{pGroundStateMeasure} near $x_0$. The following 
proposition relies on the results of Kichenassamy and Veron in \cite{KV}.
\begin{prop}\label{LimitProp}
Assume that $u$ is a solution of \eqref{pGroundStateMeasure}. Then 
$$
\lim_{x\rightarrow x_0}\frac{|u(x)-u(x_0)|}{|x-x_0|^\frac{p-n}{p-1}}
=\|u\|_\infty\left(\frac{p-1}{p-n}\right)\left(\frac{\lambda_p}{n\omega_n}\right)^{\frac{1}{p-1}}
$$
and
$$
\lim_{x\rightarrow x_0}\frac{|Du(x)|}{|x-x_0|^{\frac{p-n}{p-1}-1}}
=\|u\|_\infty\left(\frac{\lambda_p}{n\omega_n}\right)^{\frac{1}{p-1}}.
$$
Here $\omega_n$ is the Lebesgue measure of $B_1(0)\subset \R^n$.
\end{prop}
\begin{proof}
Without any loss of generality, we may assume that $u$ is positive in $\Omega$ and that $u(x_0)=1$. Recall that $u$ is $p$-harmonic in $\Omega\setminus \{x_0\}$; and in view of Proposition \ref{holderprop}, $u$ satisfies $0\le 1-u(x)\leq C|x-x_0|^\frac{p-n}{p-1}$ in $\Omega$ for some constant $C$. This permits us to use Theorem 1.1 and Remark 1.6 in \cite{KV} to conclude that there is $\gamma> 0$ such that 
\begin{equation}\label{uas}
\lim_{x\rightarrow x_0}\frac{1-u(x)}{|x-x_0|^\frac{p-n}{p-1}}
=\gamma,
\end{equation}
\begin{equation}\label{duas}
\lim_{x\rightarrow x_0}\frac{|Du(x)|}{|x-x_0|^{\frac{p-n}{p-1}-1}}
=\left(\frac{p-n}{p-1}\right)\gamma,
\end{equation}
and
\begin{equation}\label{tuas}
\lim_{x\rightarrow x_0}\left(\frac{Du(x)}{|Du(x)|}+\frac{x-x_0}{|x-x_0|}\right)=0.
\end{equation}

\par We may integrate by parts and exploit  \eqref{tuas} to get  
\begin{align*}
\lambda_p& =\int_{\Omega}|Du|^{p}dx\\
&=\lim_{r\rightarrow 0^+}\int_{\Omega\setminus B_{r}(x_0)}|Du|^{p}dx\\
&=\lim_{r\rightarrow 0^+}\int_{\Omega\setminus B_{r}(x_0)}\text{div}(u|Du|^{p-2}Du)dx\\
&=\lim_{r\rightarrow 0^+}\int_{\partial B_{r}(x_0)}u|Du|^{p-2}Du\cdot \left(- \frac{x-x_0}{|x-x_0|}\right)d\sigma\\
&=\lim_{r\rightarrow 0^+}\int_{\partial B_{r}(x_0)}u|Du|^{p-1}d\sigma.
\end{align*}
Here $\sigma$ is $n-1$ dimensional Hausdorff measure.  In view of \eqref{uas}, we actually have
\begin{equation}\label{LamPeeLimit}
\lambda_p = \lim_{r\rightarrow 0^+}\int_{\partial B_{r}(x_0)}|Du|^{p-1}d\sigma.
\end{equation}

\par Now we can apply \eqref{duas}. This limit gives 
$$
|Du(x)|^{p-1}= \left(\frac{p-n}{p-1}\gamma\right)^{p-1}|x-x_0|^{-(n-1)}+o(1)
$$
as $x\rightarrow x_0$. By \eqref{LamPeeLimit},  
$$
\lambda_p =  \lim_{r\rightarrow 0^+}\int_{\partial B_{r}(x_0)}\left[\left(\frac{p-n}{p-1}\gamma\right)^{p-1}r^{-(n-1)}+o(1)\right]d\sigma = \left(\frac{p-n}{p-1}\gamma\right)^{p-1}n\omega_n
$$
which concludes the proof. 
\end{proof}
\begin{rem}\label{potentialLim1}
A function $w\in W^{1,p}_0(\Omega)$ that satisfies 
\begin{equation}\label{PotentialPDE}
\begin{cases}
-\Delta_pw=0,& \quad x\in\Omega\setminus\{x_0\}\\
\hspace{.33in}w=1,& \quad x=x_0\\
\hspace{.33in}w=0,& \quad x\in \partial\Omega
\end{cases}
\end{equation}
weakly is called a {\it potential} function.  Observe that every extremal is a multiple of a potential function but not vice versa. For instance, if $\Omega=B_1(0)$, then $w$ is an extremal if and only if $x_0=0$.  

\par The strong maximum principle for $p$-harmonic functions implies that $0<w<1$ in $\Omega\setminus\{x_0\}$. In particular, $w$ is uniquely maximized at $x_0$.   Using similar arguments as in the proof of Proposition \ref{LimitProp}, one can easily show that 
$$
-\Delta_pw=\lambda \delta_{x_0}
$$
in $\Omega$ where $$
\lambda:=\int_\Omega|Dw|^pdx.
$$
Therefore, the conclusion of Proposition \ref{LimitProp} holds for $w$ with $\lambda$ replacing $\lambda_p$.
\end{rem}

\section{Support function of an extremal}\label{SuppSec}
Suppose now that $\Omega$ is convex and that $u\in W^{1,p}_0(\Omega)$ is a positive extremal which achieves its maximum at $x_0\in \Omega$.  By Corollary \ref{sol}, the results of Lewis \cite{JLewis} imply that
\begin{equation}\label{CapacitaryFun}
\begin{cases}
\text{ $-(x-x_0)\cdot Du(x)>0$ for  $x\in\Omega \setminus \{x_0\}$,} \\
\text{ $u$ is locally real analytic in $\Omega\setminus \{x_0\}$,}\\
\text{$\{u>t\}$ is convex for each $t\in \R$, and}\\
\text{$\{u=t\}$ has positive Gaussian curvature for each $t\in (0,\|u\|_\infty)$}. 
\end{cases}
\end{equation}
By the implicit function theorem, it also follows that the level sets of $u$ are smooth. 

\par We define the {\it support function} of $u$ as
\begin{equation}\label{SuppExtreme}
h(\xi,t):=\sup\left\{x\cdot \xi: x\in \overline{\Omega}, u(x)\ge t\right\}
\end{equation}
$\xi\in \R^n, t\in [0,\|u\|_\infty]$. For $t\in [0,\|u\|_\infty]$, $h(\cdot,t)$ is the usual support function of the convex set $\{u\ge t\}$; if $u(0)\ge t$ and $|\xi|=1$, $h(\xi,t)$ represents the distance from the origin to the hyperplane that supports $\{u\ge t\}$ with outward normal $\xi$. It follows from \eqref{CapacitaryFun} and Theorem 4 of \cite{ColSal}, $h\in C^\infty(\R^n\setminus\{0\}\times(0,\|u\|_\infty))$ and $h_t<0$.

\par Suppose $u(x_0)=\|u\|_\infty$. Then for 
$$
x\in \Omega\setminus\{x_0\},\quad\xi=-\frac{Du(x)}{|Du(x)|},\quad \text{and}\quad t=u(x),
$$
we have
$$
h(\xi,t)=x\cdot \xi,\quad h_t(\xi,t)=-\frac{1}{|Du(x)|},\quad \text{and}\quad D_\xi h(\xi,t)=x.
$$
See \cite{LonSal}. In particular, since $\xi$ is the outward unit normal to the hypersurface $\{u=t\}$ at the point $x$,  $D_\xi h(\xi,t)$ is the inverse image of the Gauss map at $x$. Moreover, as $D^2_\xi h(\xi,t)\xi=0$, the restriction of the linear transformation $D^2_\xi h(\xi,t):\R^n\rightarrow\R^n$ to $\xi^\perp:=\{z\in \R^n: z\cdot \xi=0\}$ is the inverse of the second fundamental form of $\{u=t\}$ at the point $x$ (see Section 2.5 of \cite{RS} for more on this point). In particular, $D^2_\xi h(\xi,t)|_{\xi^\perp}:\xi^\perp\rightarrow \xi^\perp$ is positive definite and its eigenvalues are the reciprocals of the principle curvatures of $\{u=t\}$ at $x$. 

\par Recall that $-\Delta_pu=0$ in $\Omega\setminus\{x_0\}$. Using this equation, Colesanti and Salani proved (in Proposition 1 of \cite{ColSal}) that $h$ satisfies 
\begin{equation}\label{plaph}
h_t^2\tr\left[\left(D^2_\xi h|_{\xi^\perp}\right)^{-1}\right]+(p-1)\left( \left(D^2_\xi h|_{\xi^\perp}\right)^{-1} \nabla_\xi h_t\cdot  \nabla_\xi h_t-h_{tt}\right)=0
\end{equation}
for each $|\xi|=1$ and $t\in (0,\|u\|_\infty)$. Here $\nabla_\xi h_t:=(I_n-\xi\otimes\xi)D_\xi h_t$ is the projection of the gradient of $h_t$ onto $\xi^\perp$. Equation \eqref{plaph} will 
have an important role in our proof of Theorem \ref{mainThm}.

\section{Convex domains}\label{UniquenessArgs}
Throughout this section, we will assume that $\Omega\subset \R^n$ is a bounded convex domain. We will also suppose that $u_0, u_1\in W^{1,p}_0(\Omega)$ are positive extremal functions which satisfy 
$$
\begin{cases}
u_0(x_0)=\|u_0\|_\infty=1\\
u_1(x_1)=\|u_1\|_\infty=1
\end{cases}
$$
for some $x_0,x_1\in \Omega$.  We aim to show that 
\begin{equation}\label{CrucialEquality}
u_0\equiv u_1.
\end{equation} 
It is easy to see that Theorem \ref{mainThm} follows from \eqref{CrucialEquality}. 

\par For each $\rho\in (0,1)$, we define the {\it Minkowski combination} of $u_0$ and $u_1$
$$
u_\rho(z):=\sup\left\{\min\{u_0(x),u_1(y)\}: z=(1-\rho)x+\rho y,\; x,y\in \overline{\Omega}\right\}
$$
$z\in \overline{\Omega}$. We recall that $u_0$ and $u_1$ are quasiconcave. Using the definition above, it is straightforward to verify 
\begin{equation}\label{SuperLevelSetIdentity}
\{u_\rho>t\}=(1-\rho)\{u_0>t\}+\rho \{u_1>t\}
\end{equation}
for each $t\in \R$. Here the addition is the usual Minkowski addition of convex sets.  In particular, $u_\rho$ itself is quasiconcave.

\par The Minkowski combination was introduced in work of Borell in \cite{Borell} when he studied capacitary functions; although his work was motivated by the previous papers of Lewis \cite{JLewis} and Gabriel \cite{Gabriel}.   We also were particularly inspired to utilize the Minkowski combination after we became aware of the work of Colesanti and Salani in \cite{ColSal}, who verified a Brunn-Minkowski inequality for $p$-capacitary functions $(1<p<n)$, and the work of Cardaliaguet and Tahraoui in \cite{CarTah} on the strict concavity of the harmonic radius. 

\par Along the way to proving \eqref{CrucialEquality}, we will need some other useful properties of $u_\rho$.
\begin{prop}\label{uRhoProp} Define 
$$
x_\rho:=(1-\rho)x_0+\rho x_1.
$$
Then the following hold:\\
\noindent $(i)$ $u_\rho(x_\rho)=\|u_\rho\|_\infty=1$.\\
$(ii)$ $u_\rho|_{\partial\Omega}=0$.\\
$(iii)$ $u_\rho\in C^\infty(\Omega\setminus\{x_\rho\})\cap C(\overline{\Omega})$.\\
$(iv)$ For each $z\in \Omega\setminus\{x_\rho\}$, there are $x\in \Omega\setminus\{x_0\}$ and $y\in \Omega\setminus\{x_1\}$ such that 
\begin{align*}
z&=(1-\rho)x+\rho y,&\\
u_\rho(z)&=u_0(x)=u_1(y),&\\
\frac{Du_\rho(z)}{|Du_\rho(z)|}&=\frac{Du_0(x)}{|Du_0(x)|}=\frac{Du_1(y)}{|Du_1(y)|},&\\
\frac{1}{|Du_\rho(z)|}&=(1-\rho)\frac{1}{|Du_0(x)|}+\rho\frac{1}{|Du_1(y)|}, &\\
\frac{D^2u_\rho(z)}{|Du_\rho(z)|^3}&\ge (1-\rho)\frac{D^2u_0(x)}{|Du_0(x)|^3}+\rho\frac{D^2u_1(y)}{|Du_1(y)|^3}.&
\end{align*}
\end{prop}
We omit the proof of the above proposition. However, we remark that $(i)$ and $(ii)$ are elementary; Theorem 4 of \cite{ColSal} and Theorem 1 of \cite{JLewis} together imply $(iii)$; and $(iv)$ follows from Section 2 of \cite{CarTah} or Section 7 of \cite{LonSal}. Using these properties we will verify that $u_\rho$ itself is an extremal for each $\rho\in (0,1)$.

\begin{lem}
$u_\rho$ is extremal. 
\end{lem}
\begin{proof}
We first show that $u_\rho$ is $p$-subharmonic and integrate by parts to derive an upper bound on the integral $\int_\Omega|Du_\rho|^pdz$. Then we show that $u_\rho$ satisfies the limits in Proposition \ref{LimitProp} (that are also satisfied by every extremal function). Finally, we combine the upper bound and limits to arrive at the desired conclusion.

\par 1. Let $z\in \Omega\setminus\{x_\rho\}$, and select $x\in \Omega\setminus\{x_0\}$ and $y\in \Omega\setminus\{x_1\}$ such that 
$$
e:=\frac{Du_\rho(z)}{|Du_\rho(z)|}=\frac{Du_0(x)}{|Du_0(x)|}=\frac{Du_1(y)}{|Du_1(y)|}
$$ 
and  
$$
\frac{D^2u_\rho(z)}{|Du_\rho(z)|^3}\ge (1-\rho)\frac{D^2u_0(x)}{|Du_0(x)|^3}+\rho\frac{D^2u_1(y)}{|Du_1(y)|^3}.
$$
Recall that such $x,y$ exist by Proposition \ref{uRhoProp}.  We have
\begin{align*}
|Du_\rho(z)|^{-(p+1)}\Delta_pu_\rho(z)&=\frac{\Delta u_\rho(z)}{|Du_\rho(z)|^3}+(p-2)\frac{D^2u_\rho(z)e\cdot e}{|Du_\rho(z)|^3}\\
&=\left(I_n+(p-2)e\otimes e\right)\cdot \frac{D^2u_\rho(z)}{|Du_\rho(z)|^3}.
\end{align*}
Note that $\min\{1,p-1\}>0$ is a lower bound on the eigenvalues of the matrix $I_n+(p-2)e\otimes e$. Therefore, 
\begin{align*}
|Du_\rho(z)|^{-(p+1)}\Delta_pu_\rho(z)&\ge\left(I_n+(p-2)e\otimes e\right)\cdot\left( (1-\rho)\frac{D^2u_0(x)}{|Du_0(x)|^3}+\rho\frac{D^2u_1(y)}{|Du_1(y)|^3}\right)\\
&= (1-\rho)|Du_0(x)|^{-(p+1)}\Delta_pu_0(x)+ \rho|Du_1(y)|^{-(p+1)}\Delta_pu_1(y)\\
&= (1-\rho)\cdot 0+ \rho\cdot 0\\
&=0.
\end{align*}
Consequently, $-\Delta_p u_\rho\le 0$ in $\Omega\setminus\{x_\rho\}$.

\par 2. The divergence theorem gives
$$
\int_{\Omega\setminus B_r(x_\rho)}\text{div}(u_\rho|Du_\rho|^{p-2}Du_\rho)dz =\int_{\partial B_r(x_\rho)}u_\rho|Du_\rho|^{p-2}Du_\rho \cdot \left(-\frac{z-x_\rho}{|z-x_\rho|}\right)d\sigma.
$$
On the other hand, since $u_\rho$ is a positive $p$-subharmonic function in $\Omega\setminus\{x_\rho\}$
\begin{align*}
\int_{\Omega\setminus B_r(x_\rho)}\text{div}(u_\rho|Du_\rho|^{p-2}Du_\rho)dz &=\int_{\Omega\setminus B_r(x_\rho)}\left(u_\rho\Delta_p u_\rho +|Du_\rho|^{p}\right)dz \\
&\ge \int_{\Omega\setminus B_r(x_\rho)}|Du_\rho|^{p}dz.
\end{align*}
As $u_\rho\le 1$,
\begin{equation}\label{UpperLiminf}
\int_{\Omega}|Du_\rho|^pdz\le \liminf_{r\rightarrow 0^+}\int_{\partial B_r(x_\rho)}|Du_\rho|^{p-1}d\sigma. 
\end{equation}

\par 3. Let $w_\rho$ be a solution of the PDE \eqref{PotentialPDE} with $x_\rho$ replacing $x_0$. As $u_\rho$ is $p$-subharmonic, $u_\rho(x_\rho)=1$ and $u_\rho|_{\partial\Omega}=0$,  weak comparison implies $u_\rho\le w_\rho.$ This is a version of Borell's inequality; see \cite{Borell, CarTah}. In particular, 
\begin{equation}\label{WlowerBound}
w_\rho(z)\ge u_\rho(z)\ge \min\{u_0(x),u_1(y)\}
\end{equation}
whenever $z=(1-\rho)x+\rho y$. 

\par  Now let $z^k\rightarrow x_\rho$ with $z^k\neq x_\rho$ for all $k\in \N$ sufficiently large.  Set
$$
\begin{cases}
x^k:=z^k+(x_0-x_\rho) \\
y^k:=z^k+(x_1-x_\rho) 
\end{cases}
$$
for each $k\in \N$. Observe $z^k=(1-\rho)x^k+\rho y^k$ and  
$$
|z^k-x_\rho|=|x^k-x_0|=|y^k-x_1|.
$$
Setting $\lambda:= \int_\Omega|Dw_\rho|^pdz$, we have from Proposition \ref{LimitProp}, Remark \ref{potentialLim1} and \eqref{WlowerBound} that
\begin{align*}
\left(\frac{p-1}{p-n}\right)\left(\frac{\lambda}{n\omega_n}\right)^\frac{1}{p-1}
&=\lim_{k\rightarrow \infty}
\frac{1-w_\rho(z^k)}{|z^k-x_\rho|^{\frac{p-n}{p-1}}}\\
&\le \lim_{k\rightarrow \infty} \max\left\{\frac{1-u_0(x^k)}{|x^k-x_0|^{\frac{p-n}{p-1}}},
\frac{1-u_1(y^k)}{|y^k-x_1|^{\frac{p-n}{p-1}}}\right\}\\
&=\left(\frac{p-1}{p-n}\right)\left(\frac{\lambda_p}{n\omega_n}\right)^\frac{1}{p-1}.
\end{align*}
It follows that $\lambda= \lambda_p$. In view of \eqref{WlowerBound}, and since the sequence $z^k$ was arbitrary,
\begin{equation}\label{urhoLimit1}
\lim_{z\rightarrow x_\rho}
\frac{1-u_\rho(z)}{|z-x_\rho|^{\frac{p-n}{p-1}}}=\left(\frac{p-1}{p-n}\right)\left(\frac{\lambda_p}{n\omega_n}\right)^\frac{1}{p-1}.
\end{equation}

\par 4. Again let $z^k\rightarrow x_\rho$ with $z^k\neq x_\rho$ for all $k\in \N$ sufficiently large. By Proposition \ref{uRhoProp}, there are 
$x^k\in \Omega\setminus\{x_0\}$ and $y^k\in \Omega\setminus\{x_1\}$ such that $z^k=(1-\rho)x^k+\rho y^k$, 
\begin{equation}\label{AllfunctionEqualk}
u_\rho(z^k)=u_0(x^k)=u_1(y^k),
\end{equation}
and
\begin{equation}\label{ReciprocalDerivativeEqualk}
\frac{1}{|Du_\rho(z^k)|}=(1-\rho)\frac{1}{|Du_0(x^k)|}+\rho\frac{1}{|Du_1(y^k)|}.
\end{equation}
Since $u_\rho(z^k)\rightarrow 1$, \eqref{AllfunctionEqualk} implies that $x^k\rightarrow x_0$ and $y^k\rightarrow x_1$ as $u_0$ and $u_1$ are uniquely maximized as these points, respectively.  Combining this fact with Proposition \ref{LimitProp}, \eqref{urhoLimit1} and again with \eqref{AllfunctionEqualk} also gives 
\begin{equation}
\label{yjlim}
\lim_{k\to\infty}\frac{|y^k-x_1|}{|z^k-x_\rho|}=\lim_{k\to\infty}\frac{|x^k-x_0|}{|z^k-x_\rho|}=1.
\end{equation}

\par By \eqref{ReciprocalDerivativeEqualk},
\begin{align*}
\frac{|z^k-x_\rho|^{\frac{p-n}{p-1}-1}}{|Du_\rho(z^k)|}&=(1-\rho)\frac{ |z^k-x_\rho|^{\frac{p-n}{p-1}-1}}{|Du_0(x^k)|}+\rho\frac{|z^k-x_\rho|^{\frac{p-n}{p-1}-1}}{|Du_1(y^k)|}\\
&=\left(\frac{|z^k-x_\rho|}{|x^k-x_0|}\right)^{\frac{p-n}{p-1}-1}(1-\rho)\frac{ |x^k-x_0|^{\frac{p-n}{p-1}-1}}{|Du_0(x^k)|}+
\left(\frac{|z^k-x_\rho|}{|y^k-x_1|}\right)^{\frac{p-n}{p-1}-1}\rho\frac{|y^k-x_1|^{\frac{p-n}{p-1}-1}}{|Du_1(y^k)|}.
\end{align*}
We can now employ the second limit in Proposition \ref{LimitProp} and \eqref{yjlim} to obtain 
$$
\lim_{k\rightarrow \infty}
\frac{|Du_\rho(z^k)|}{|z^k-x_\rho|^{\frac{p-n}{p-1}-1}}=\left(\frac{\lambda_p}{n\omega_n}\right)^\frac{1}{p-1}.
$$
And since $z^k$ was arbitrary,  
\begin{equation}\label{urhoLimit2}
\lim_{z\rightarrow x_\rho}
\frac{|Du_\rho(z)|}{|z-x_\rho|^{\frac{p-n}{p-1}-1}}=\left(\frac{\lambda_p}{n\omega_n}\right)^\frac{1}{p-1}.
\end{equation}

\par 5. Using the upper bound \eqref{UpperLiminf} and the limits \eqref{urhoLimit1} and \eqref{urhoLimit2}, we can 
proceed with the same arguments as in the proof of Proposition \ref{LimitProp} to conclude
\begin{align*}
\int_{\Omega}|Du_\rho|^pdz&\le \liminf_{r\rightarrow 0^+}\int_{\partial B_r(x_\rho)}|Du_\rho|^{p-1}d\sigma= \lambda_p.
\end{align*}
\end{proof}
In order to verify \eqref{CrucialEquality}, we will employ the respective support functions $h_0$, $h_1$, and $h_\rho$ of $u_0$, $u_1$ and $u_\rho$; recall the support function of an extremal was defined in \eqref{SuppExtreme}. In particular, we note that the identity \eqref{SuperLevelSetIdentity} implies 
$$
h_\rho=(1-\rho)h_0+\rho h_1. 
$$
Using this identity with the fact that $h_0, h_1$ and $h_\rho$ all satisfy equation \eqref{plaph}, Colesanti and Salani showed for each $t\in (0,1)$ there is $C(t)>0$ such that
\begin{equation}\label{HesshmatrixId}
D^2_\xi h_0(\xi,t)|_{\xi^\perp}=C(t)D^2_\xi h_1(\xi,t)|_{\xi^\perp}
\end{equation}
and
\begin{equation}\label{DthId}
\partial_th_0(\xi,t)=C(t)\partial_th_1(\xi,t)
\end{equation}
for all $|\xi|=1$ (see the proof of Theorem 1 in \cite{ColSal}).

\par As $D^2_\xi h_0(\xi,t)\xi=D^2_\xi h_0(\xi,t)\xi=0\in \R^n$, it follows from \eqref{HesshmatrixId} and the homogeneity of $h_0$ and $h_1$ that 
$$
D^2_\xi h_0(\xi,t)=C(t)D^2_\xi h_1(\xi,t)
$$
for $(\xi,t)\in (\R^n\setminus\{0\})\times(0,1)$. Upon integration we find 
$$
h_0(\xi,t)=C(t)h_1(\xi,t)+a(t)\cdot \xi +b(t)
$$
for some $a(t)\in \R^n$ and $b(t)\in \R$. Since $h_0$ and $h_1$ are homogeneous of degree one, it must be that $b(t)=0$ for each $t\in (0,1)$. Thus,  
$$
h_0(\xi,t)=C(t)h_1(\xi,t)+a(t)\cdot \xi.
$$
Taking the time derivative of both sides of this equation gives
$$
\partial_th_0(\xi,t)=C(t)\partial_th_1(\xi,t)+C'(t)h_1(\xi,t)+a'(t)\cdot \xi.
$$
Comparing with \eqref{DthId} leads us to
$$
C'(t)h_1(\xi,t)+a'(t)\cdot \xi=0.
$$

\par Suppose that $C'(t_0)\neq 0$ for some $t_0\in (0,1)$. Then 
$$
h_1(\xi,t_0)=\left[\frac{-a'(t_0)}{C'(t_0)}\right]\cdot \xi.
$$
This would imply the level set $\{u_1=t_0\}$ is the singleton $\{-a'(t_0)/C'(t_0)\}$, which is not possible. Therefore, $C'(t)= 0$ for $t\in (0,1)$ and thus $a'(t)\cdot \xi=0$ for all $\xi$. Consequently, $a'(t)=0$ for $t\in (0,1)$. Since $h_0$ and $h_1$ coincide at $t=0$, $C(t)=1$ and $a(t)=0$ for all $t\in [0,1]$. As a result, $h_0\equiv h_1$ and so $\{u_0\ge t\}=\{u_1\ge t\}$ for each $t\in [0,1]$ (Theorem 8.24 in \cite{RocWet}). This verifies \eqref{CrucialEquality}.

\begin{proof}[~Proof of Corollary \ref{Greenthm}] By equation \eqref{GreenPDE} and inequality \eqref{SecondMorrey},
\begin{align*}
G(y,y)&=\int_\Omega |DG(x,y)|^pdx\\
&\ge\lambda_p\|G(\cdot,y)\|^p_\infty\\
&\ge\lambda_pG(y,y)^p.
\end{align*}
Therefore, $G(y,y)\le\lambda_p^{-1/(p-1)}$ and equality holds if and only if $G(\cdot,y)$ is extremal.  Theorem \ref{mainThm} in turn implies that equality occurs if and only if $y=x_0$. 
\end{proof}

\section{Nonuniqueness}\label{CounterEx}
We will now explain that uniqueness does not hold for general domains by providing a few explicit examples.  These instances include planar  annuli, bow tie and dumbbell shaped domains.  The perceptive reader will also see how to construct other examples from our remarks below. 
\begin{figure}[h]
\begin{center}
\includegraphics[width=.5\textwidth]{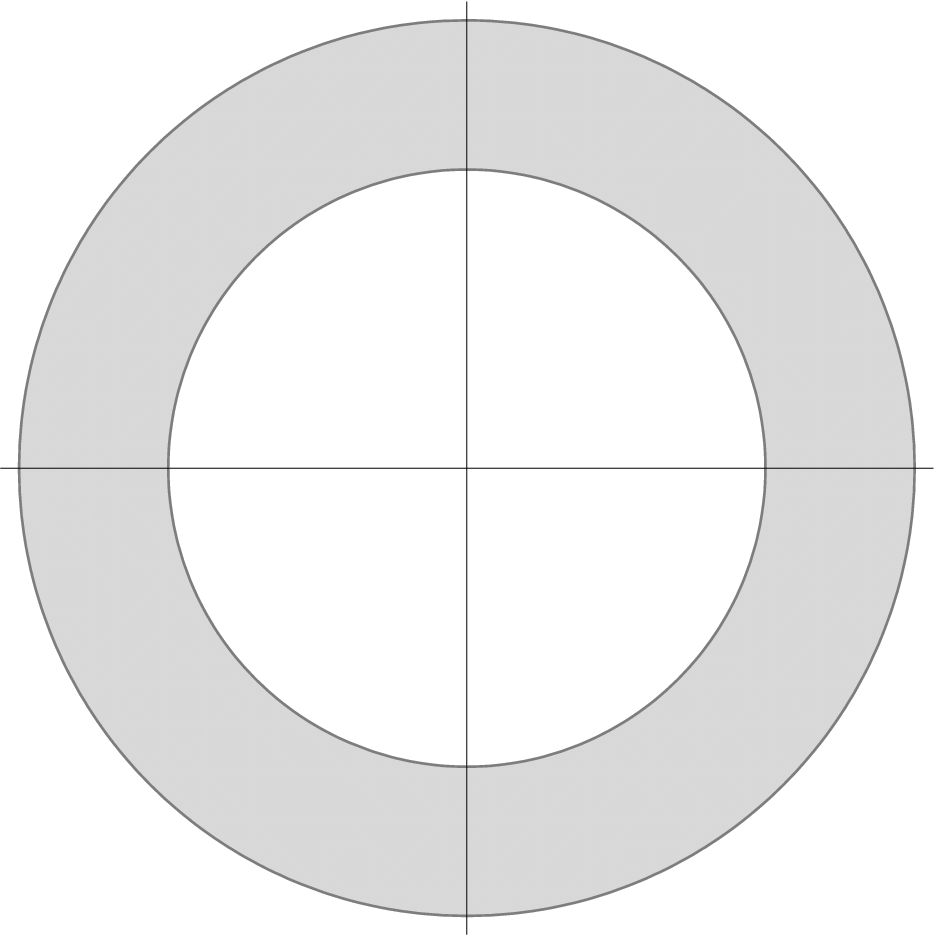}
\caption{$\Omega_{4,6}$ when $n=2$.}
\label{AnnFig}
\end{center}
\end{figure}
\begin{ex}
Define 
$$
\Omega_{r_1,r_2}:=\{x\in \R^n: r_1<|x|<r_2\}
$$
for $r_1,r_2>0$ with $r_1<r_2$.  
As mentioned above, there is a positive extremal $u$ that achieves is maximum at a single point $x_0\in \Omega_{r_1,r_2}$.  Notice that for any $n\times n$ orthogonal matrix $O$,  $v:=u\circ O$ is a positive extremal and $\|v\|_\infty=\|u\|_\infty$. Consequently, for each $y_0\in \Omega_{r_1,r_2}$ with $|y_0|=|x_0|$, there is a distinct positive extremal with supremum norm equal to $\|u\|_\infty$. Thus, uniqueness of extremals does not hold for annuli as showed in Figure \ref{AnnFig}.  
\end{ex}

\begin{ex}\label{BowTieEx} Consider the ``bow tie" domain in the plane 
$$
\Omega_\epsilon:=\left\{(x_1,x_2)\in \R^2: |x_2|<|x_1|+\epsilon, |x_1|<1\right\}
$$  
for $\epsilon>0$. Note, in particular, that $\Omega_\epsilon$ is star-shaped with respect to the origin; see Figure \ref{BowtieFig}. 
\begin{figure}[h]
\begin{center}
\includegraphics[width=.5\textwidth]{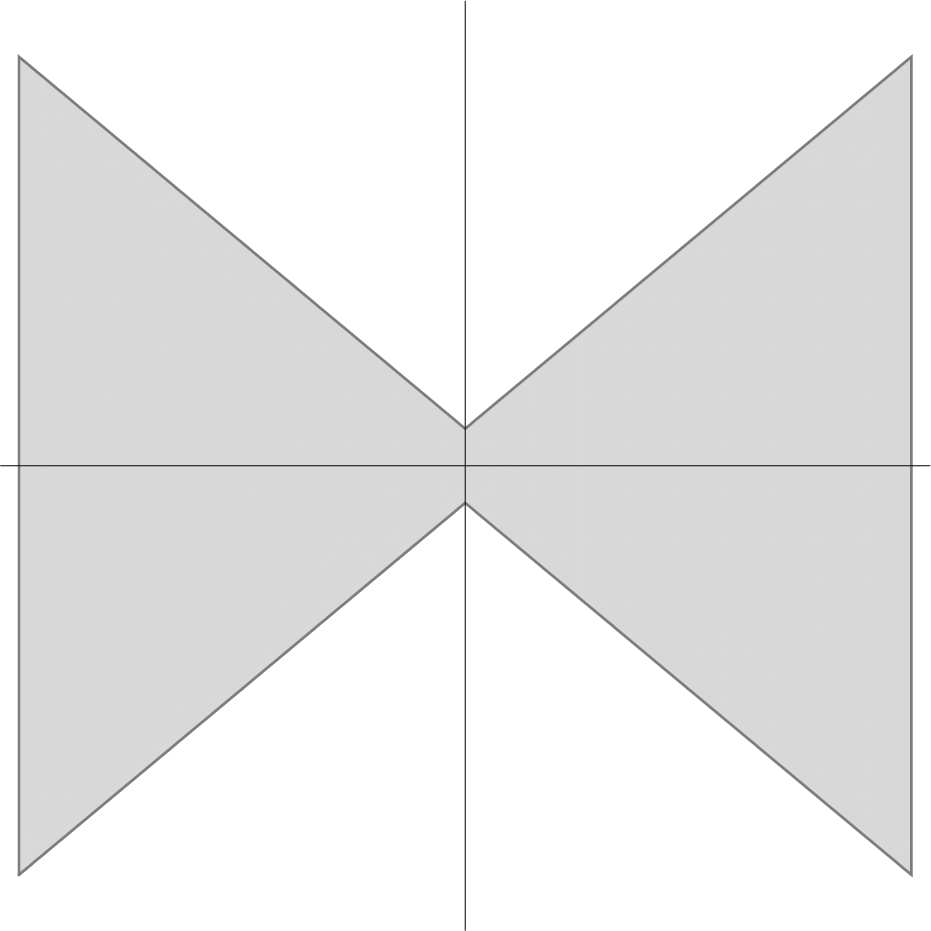}
\caption{$\Omega_{1/10}$}
\label{BowtieFig}
\end{center}
\end{figure}
Let $u_\epsilon$ be a positive extremal for $\Omega_\epsilon$ with $\|u_\epsilon\|_\infty=1$. If $u_\epsilon$ is unique, then it must be that 
\begin{equation}\label{uisoneatZEro}
u_\epsilon(0,0)=1.
\end{equation}
This is due to the fact that the $\Omega_\epsilon$ and the $p$-Laplacian are invariant with respect to reflection about the $x_1$ and $x_2$ axes. 

\par Let us assume \eqref{uisoneatZEro} holds for each $\epsilon>0$ and extend $u_\epsilon$ to be 0 outside of $\Omega_\epsilon$. Notice that the resulting function, which we also denote as $u_\epsilon$, belongs to $W^{1,p}(\R^2)$. Also note that since $\Omega_0\subset \Omega_\epsilon$ 
$$
\int_{\R^2}|Du_\epsilon|^pdx=\int_{\Omega_\epsilon}|Du_\epsilon|^pdx=\lambda_p(\Omega_\epsilon)\le \lambda_p(\Omega_0).
$$
Consequently, there is a decreasing sequence of positive numbers $(\epsilon_j)_{j\in \N}$ tending to 0 and a continuous function $u_0:\R^2\rightarrow [0,1]$ for which $u_{\epsilon_j} \rightarrow u_0$ locally uniformly on $\R^2$. In view of \eqref{uisoneatZEro}, 
$$
u_0(0,0)=1.
$$
On the other hand, $u_\epsilon(0,2\epsilon)=0$ for all $\epsilon>0$. Thus 
$$
u_0(0,0)=\lim_{j\rightarrow\infty}u_{\epsilon_j}(0,2\epsilon_j)=0,
$$
which is a contradiction. 
\par As a result, we conclude that there is some $\epsilon>0$ such that $u_\epsilon$ does not achieve its maximum value at $(0,0)$. For this value of $\epsilon$, $\Omega_\epsilon$ will have a least two  positive extremals with supremum norm equal to 1. 
\end{ex}
\begin{ex}
The same ideas used in Example \ref{BowTieEx}, can be used to show the dumbbell-shaped domain
$$
B_1(-5,0)\cup \left([-5,5]\times[-\delta,\delta]\right)\cup B_1(5,0)\quad (0<\delta<1)
$$
does not have unique extremals for some $\delta>0$ chosen small enough. See Figure \ref{DumbellFig}.  
\begin{figure}[h]
\begin{center}
\includegraphics[width=.8\textwidth]{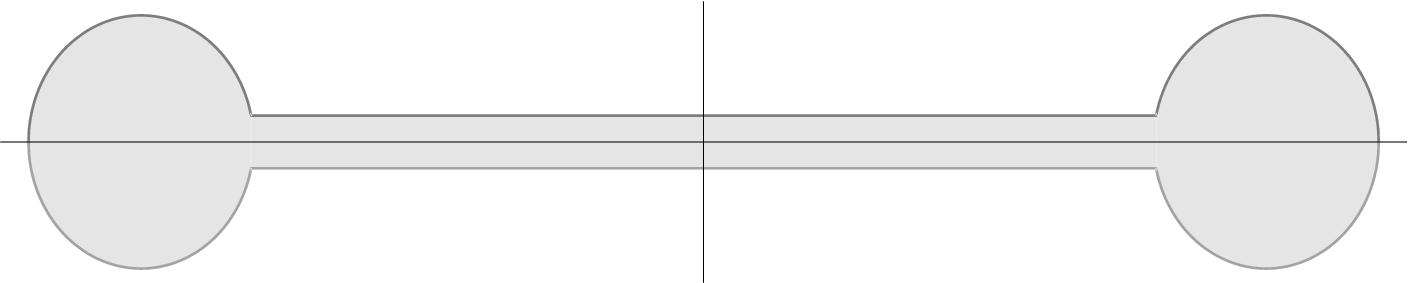}
\caption{Dumbbell-shaped domain for $\delta=1/5$.}
\label{DumbellFig}
\end{center}
\end{figure}
\end{ex}

\section{Steiner symmetric domains}\label{Steiner}
Theorem \ref{mainThm} implies that if a convex domain has some reflectional symmetry, then we have additional information on the location of the maximum points of positive extremals.  More precisely, we can make the following observation. 

\begin{cor}
Assume $\Omega\subset\R^n$ is a convex domain that is invariant with respect to reflection across the hyperplanes $\{x\in \R^n: x_j=0\}$ for $j=1,\dots, n$.  Then any positive (negative) extremal achieves 
its maximum (minimum) value at $0\in \R^n$. 
\end{cor}
\begin{proof}
Assume $u\in W^{1,p}_0(\Omega)$ is a positive extremal that achieves it maximum value at $z$.  As $\Omega$ is invariant with respect to $\{x\in \R^n: x_1=0\}$, the function
$$
u_1(x_1,x_2,\dots, x_n):=u(-x_1,x_2,\dots, x_n),\quad (x_1,x_2,\dots, x_n)\in \Omega
$$
belongs to $W^{1,p}_0(\Omega)$ and $\|u\|_\infty=\|u_1\|_\infty$. Moreover, it is routine to verify that $u_1$ is also a positive extremal that achieves it maximum at the reflection of $z$ about the plane $\{x\in \R^n: x_1=0\}$.  By Theorem \ref{mainThm}, $u_1=u$ which forces $z\in \{x\in \R^n: x_1=0\}$.   Repeating this argument for $j=2,\dots, n$, we find $z\in \{x\in \R^n: x_j=0\}$ for $j=1,\dots, n$. As a result, $z=0$.
\end{proof}
We now seek to extend this observation. We will show below that certain symmetric two dimensional domains have unique extremals without assuming the domains were convex to begin with. To this end, we employ Steiner symmetrization. In particular, we will make use of the results by Cianchi and Fusco in \cite{CianchiFusco} on the equality condition in the P\'olya-Szeg\"o inequality associated with  Steiner symmetrization. We also use special properties of the critical points of $p$-harmonic functions in two dimensions due to Manfredi in \cite{Man88}. 

\par Let us first briefly recall the notion of the Steiner symmetrization of a subset of $\R^2$. For a given $A\subset\R^2$ and $a\in \R$, we will denote $A\cap \{x_1=a\}$ as the intersection of $A$ with the vertical line $x_1=a$.  We also will write ${\cal L}^m$ for the outer Lebesgue measure defined on all subsets of $\R^m$ $(m=1,2)$. 
\begin{defn}
Assume $A\subset \R^2$. The {\it Steiner symmetrization} of $A$ with respect to the $x_1$ axis is 
$$
A^{*}_1=\left\{(a,b)\in \R^2: |b|< \frac{1}{2} {\cal L}^1\left(A\cap \{x_1=a\}\right)\right\}.
$$ 
$A$ is said to be {\it Steiner symmetric} with respect to the $x_1$ axis if $A^*_1=A$. 
\end{defn} 

Now suppose $u:\R^2\rightarrow [0,\infty)$ is Lebesgue measurable. We can use the above definition to provide the following rearrangement of $u$ 
$$
u^*_1(x):=\int^\infty_0\chi_{\{u>t\}^{*}_{1}}(x)dt, \quad x\in \R^2.
$$
This function is called the {\it Steiner rearrangement} of $u$ with respect to the $x_1$ axis. Observe that 
\begin{equation}\label{EquiMeasProp}
\{u^*_1>t\}=\{u>t\}^*_1
\end{equation}
for each $t\ge 0$. Note also that $u^*_1(x_1,\cdot)$ and $u(x_1,\cdot)$ have the same distribution for ${\cal L}^1$ almost every $x_1\in \R$.  

\par It is known that if $p\in [1,\infty)$, $\Omega\subset\R^2$ is a bounded domain and $u\in W^{1,p}_0(\Omega)$, then $u^*_1\in W^{1,p}_0(\Omega)$. Moreover, if $\Omega^*_1$ is 
Lebesgue measurable, the {\it P\'olya-Szeg\"{o} inequality}
\begin{equation}\label{PSineq}
\int_{\Omega^*_1} |D u_1^*|^p dx \leq \int_\Omega |D u|^p dx 
\end{equation}
holds, see \cite{Brock, CianchiFusco}. Cianchi and Fusco showed that if $\Omega^*_1=\Omega$ and equality holds
in \eqref{PSineq}, then $u^*_1=u$ provided 
\begin{equation}\label{CFcondition}
{\cal L}^2\left(\{u_{x_2}=0\}\right)=0
\end{equation}
(Theorem 2.2 in \cite{CianchiFusco}).  All of the above definitions and facts regarding Steiner symmetrization and rearrangements with respect to the $x_1$ axis have obvious counterparts with respect to the $x_2$ axis. 

\par Our main assertion regarding the uniqueness of extremals on Steiner symmetric domains is as follows.  
\begin{prop}\label{2DSimple} Assume $\Omega\subset \R^2$ is a bounded domain that is equal to its Steiner symmetrization with respect to the $x_1$ and $x_2$ axes. Then any positive (negative) extremal achieves 
its maximum (minimum) value at $0\in \R^2$.
\end{prop}
\begin{proof} 
Assume $u\in W^{1,p}_0(\Omega)$ is a positive extremal with $\|u\|_\infty=u(z)=1$. In view of \eqref{EquiMeasProp}, $\|u^*_1\|_\infty=1$, as well. By the P\'olya-Szeg\"o inequality \eqref{PSineq},
we easily conclude  $u^*_1$ is extremal and $\int_\Omega |D u_1^*|^p dx = \int_\Omega |D u|^p dx$.  We now claim that $u$ satisfies \eqref{CFcondition}. Once we verify this assertion, we would have $u=u^*_1$ which implies $u(x_1,x_2)=u(x_1,-x_2)$ for all $(x_1,x_2)\in \Omega$. As a result $z$ belongs to the $x_1$ axis, and very similarly we would have that $z$ also belongs to the $x_2$ axis. Therefore, $z=0\in \R^2$. 

\par  Let us now show that any positive extremal $u\in W^{1,p}_0(\Omega)$ satisfies \eqref{CFcondition}. Recall that $u$ is $p$-harmonic in $\Omega\setminus\{z\}$ and therefore, $u\in C^1_{\text{loc}}(\Omega\setminus\{z\})$. By the results of Manfredi in \cite{Man88}, we know the zeros of $Du$ are isolated in $\Omega\setminus\{z\}$. Consequently, $u$ is locally real analytic in $S:=\Omega\setminus\left(\{z\}\cup\{|Du|=0\}\right)$, which is an open set of full measure. In particular, $u_{x_2}$ is also locally real analytic in $S$.  Therefore, if  
$$
{\cal L}^2\left(\{x\in S : u_{x_2}(x)=0\}\right)>0,
$$
then it must be that $u_{x_2}\equiv 0$ in $S$; see section 3.1 of \cite{KraPar}. Since $u_{x_2}$ is continuous in $\Omega\setminus\{z\}$, it would then follow that $u_{x_2}\equiv 0$ in $\Omega\setminus\{z\}$, as well. However, this is clearly not possible as the function 
$$
[0,\infty) \ni t\mapsto u(z+te_2)
$$
is positive at $t=0$ and vanishes for all $t>0$ sufficiently large.  As a result, \eqref{CFcondition} holds and the assertion follows.  
\end{proof}
\begin{figure}[h]
\begin{center}
\includegraphics[width=.75\textwidth]{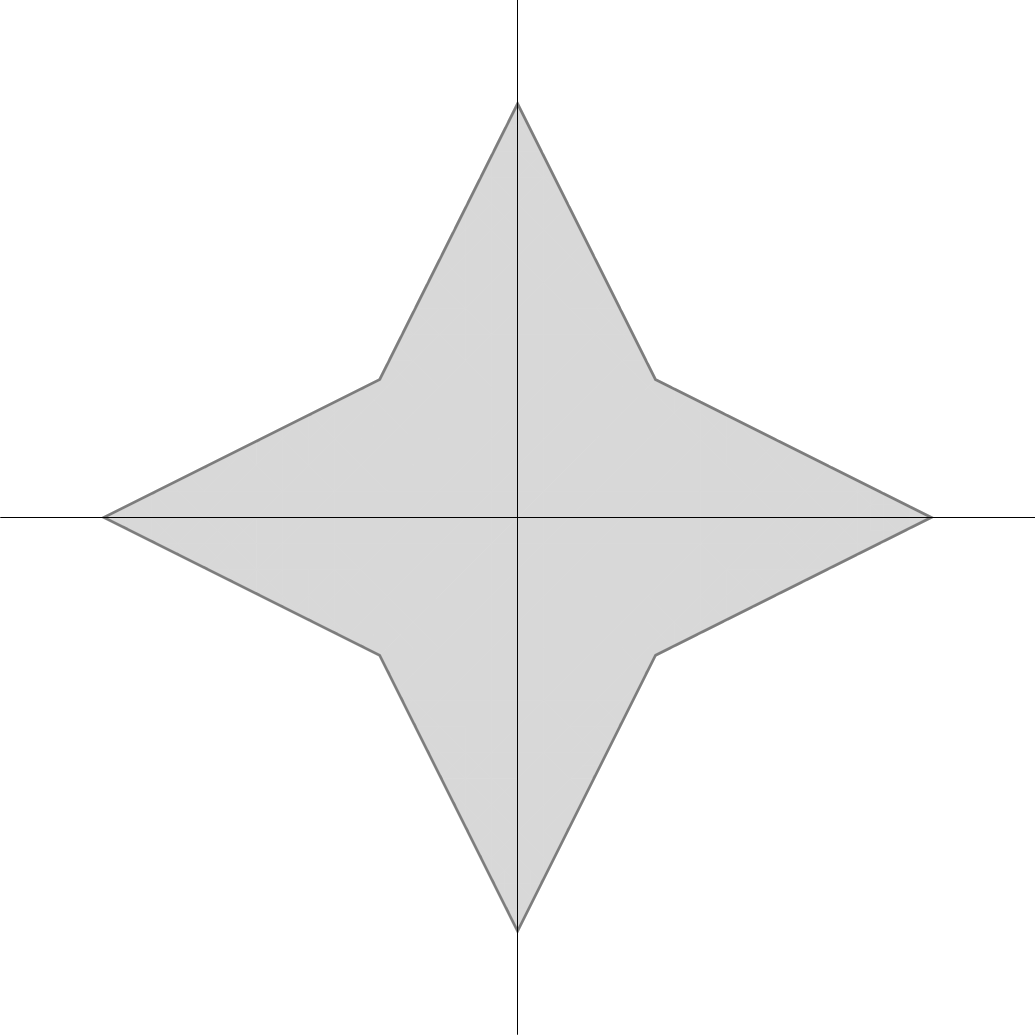}
\caption{}
\label{StarFish}
\end{center}
\end{figure}
\begin{figure}[h]
\begin{center}
\includegraphics[width=.7\textwidth]{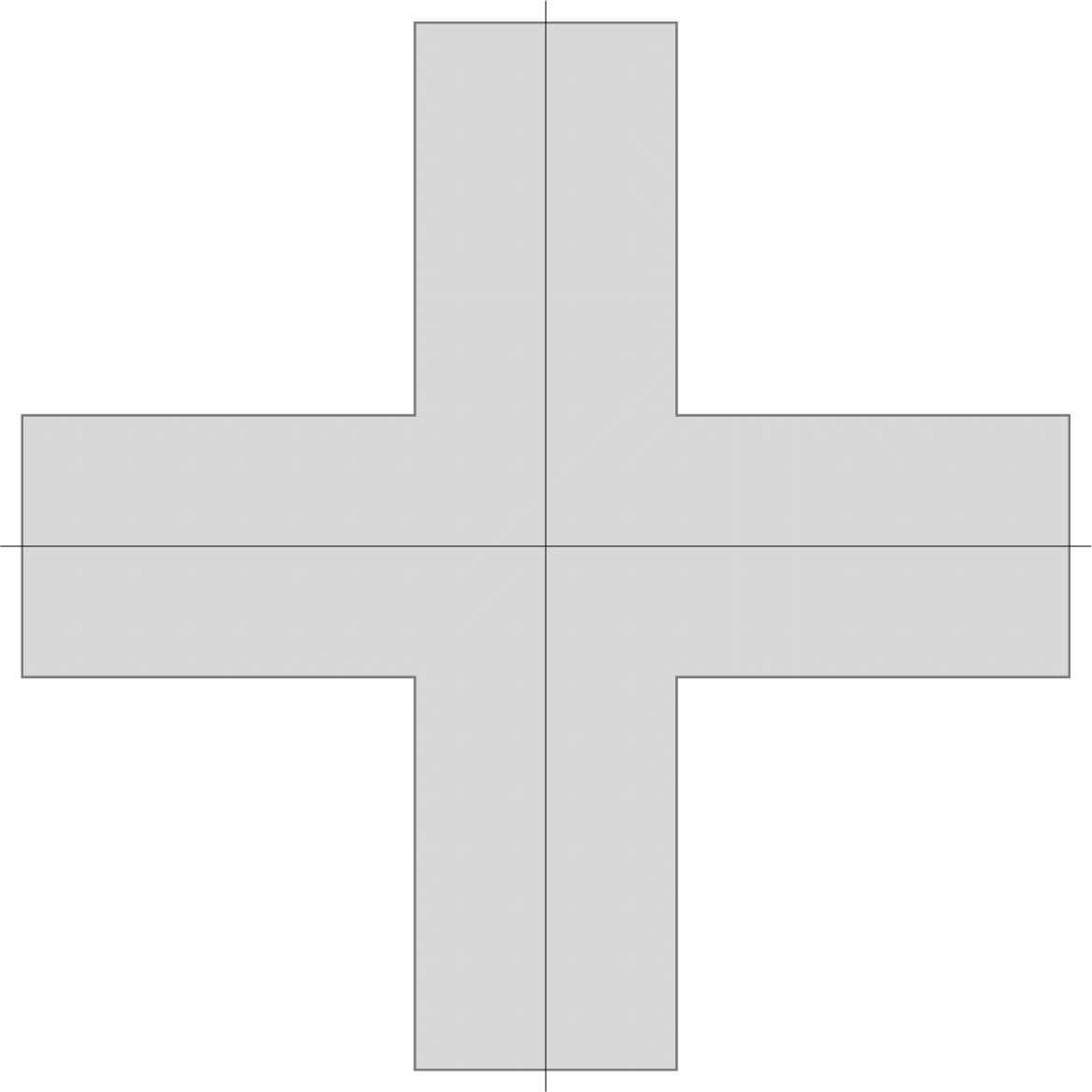}
\caption{}
\label{CrossFig}
\end{center}
\end{figure}
\begin{figure}[h]
\begin{center}
\includegraphics[width=.7\textwidth]{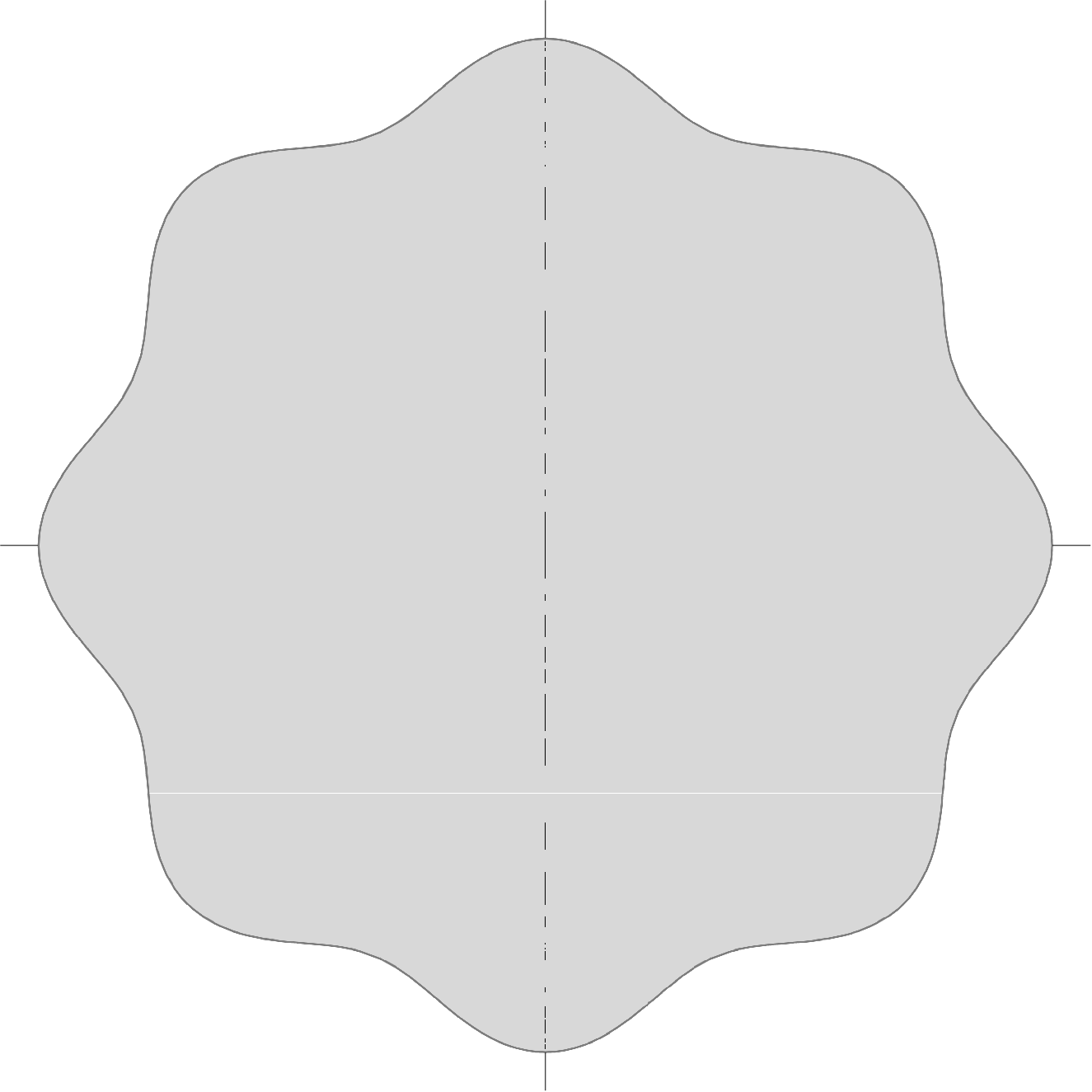}
\caption{}
\label{WobbleFig}
\end{center}
\end{figure}
See Figures \ref{StarFish}, \ref{CrossFig} and \ref{WobbleFig} for Steiner symmetric, nonconvex domains $\Omega$ for which Proposition \ref{2DSimple} applies to.  Figure \ref{StarFish} displays 
$$
\{(x_1,x_2)\in \R^2: |x_1|+2|x_2|<1, 2|x_1|+|x_2|<1 \},
$$
Figure \ref{CrossFig} shows 
$$
([-3,3]\times[-1,1])\cup([-1,1]\times[-3,3])
$$
and Figure \ref{WobbleFig} exhibits the region bounded by the curve
$$
r=10+\frac{13}{20}\cos(8\theta)
$$
given in polar coordinates.


\end{document}